\documentclass[11pt,a4paper,makeidx, amstex]{amsart}
\usepackage[latin1]{inputenc}        
\usepackage[dvips]{graphics}
\usepackage{pstricks-add}
\usepackage
{graphicx}
\usepackage{epstopdf}
\usepackage{mathrsfs}
\let\mathcal\mathscr
\usepackage{color}		
\usepackage{epsfig}
\psdraft
\usepackage{amssymb}
\usepackage{amsmath}
\usepackage{amsfonts}
\usepackage{latexsym}
\usepackage[T1]{fontenc}
\usepackage[latin1]{inputenc}
\usepackage{hyperref}
\usepackage[all,ps]{xy}
\RequirePackage{amsthm}
\RequirePackage{amssymb}
\usepackage[all,ps]{xy}
\usepackage{latexsym}
\usepackage{amscd}
\usepackage[latin1]{inputenc}
\usepackage[T1]{fontenc}
\usepackage{color}
\RequirePackage{amsthm}
\RequirePackage{amssymb}
\usepackage[all,ps]{xy}
\usepackage{latexsym}

\def\P{{\bf P}}

\def\R{{\mathbb R}}
\def\Q{{\mathbb Q}}

\def\Mor{\mathop{\rm Mor}\nolimits}

\def\supp{\mathop{\rm Supp}\nolimits}

\def\tilde{\widetilde}

\def\phi{\varphi}

\def\D{{\Delta}}

\def\cC{{\mathcal C}}
\def\cE{{\mathcal E}}

\def\cL{{\mathcal L}}

\def\cO{{\mathcal O}}
\def\cP{{\mathcal P}}

\def\cR{{\mathcal R}}

\def\cV{{\mathcal V}}
\def\cW{{\mathcal W}}
\def\cX{{\mathcal X}}
\def\cY{{\mathcal Y}}

\def\dim{\mathop{\rm dim}\nolimits}

\def\Mor{\mathop{\rm Mor}\nolimits}

\newtheorem{thm}{Theorem}[section]
\newtheorem{defn}[thm]{Definition}

\newtheorem{claim}[thm]{Claim}
\newtheorem{convention}[thm]{Convention}
\newtheorem{alg}[thm]{Algorithm}

\newtheorem{rmk}[thm]{Remark}
\newtheorem{prop}[thm]{Proposition}
\newtheorem{lem}[thm]{Lemma}

\newtheorem{ex}[thm]{Example}

\title[Remarks about Bubbles]{Remarks about Bubbles}
\author[Michael McQuillan and Gianluca Pacienza]{Michael McQuillan and Gianluca Pacienza}
\date{\today}

\begin{document}
\maketitle
{\let\thefootnote\relax
\footnote{ 
\textbf{A.M.S.~classification :} 14J40. \\
This work was partially supported by the French--Italian European Research Group in Algebraic Geometry GDRE-GRIFGA. G.P. was partially supported by the ANR project "CLASS'' no. ANR-10-JCJC-0111.

 }}
\numberwithin{equation}{section}

{\bf Abstract}: {\small We make some remarks about bubbling
on, not necessarily proper, champs de Deligne-Mumford,
{\it i.e.} compactification of the space of
mappings from a given (wholly scheme like) curve, so, in
particular, on quasi-projective projective varieties. Under
hypothesis on both the interior and the boundary such as
Remark \ref{rmk:interior} below, this implies an optimal
logarithmic variant of Mori's Bend-and-Break. The main technical
remark is \ref{thm:logMM}, while our final remark, the cone theorem,
\ref{rmk:cone}, is a variant.}

%
%
%
%
\section{Introduction}
%

Gromov convergence in the setting of,
say, a compact K\"ahler manifold, $X$,
with boundary $\Delta=\sum_i\Delta_i$ is amusing. The
basic point is a lemma of Mark Green,
\cite{green}, that if $f_n\to f$ 
are maps from a Riemann surface $\Sigma$
converging uniformly
on compact sets and $D$ a divisor on $X$
such that $f_n^{-1}(D)=\emptyset$, then
this can only fail for $f$ if it factors
through $D$. Trivially,
\cite[Fact 2.2.4]{bloch}, this implies
that if $f_n:\Sigma\to X\backslash \Delta$
converges to a disc with (non-trivial) 
bubbles then there must be a map from $\mathbb{A}^1$
to $X\backslash\D$ or some other stratum,
$\cap_{i\in I}\D_i\backslash\cup_{j\notin I} \D_j$,
where for notational convenience we write
the latter as $X_I\backslash \D_J$, which
includes the former for $I=\emptyset$.     
Once the local algebra in the purely
meromorphic context is identified,
\ref{prop:2} \& \ref{prop:b&b}, this
combines with Mori's bend \& break
technique to yield some further 
amusement, {\it e.g.},

\begin{rmk}\label{rmk:interior}
Let $V$ be a smooth quasi-projective variety,
over an algebraically closed field of any
characteristic,
and suppose it admits a compactification
$V\hookrightarrow X$ such that $X$ is
projective with at worst quotient singularities
and for $\Delta_i$ the irreducible components
of the boundary $\Delta=X\backslash V$,
no stratum $X_{I}\backslash\D_J$, $I\neq\emptyset$,
admits a non-trivial map from $\mathbb{A}^1$, then
for every map from a curve
$f:C\to X$ such that $(K_X+\D)\cdot _f C<0$
to every $x\in f(C)\cap V$ there is a
rational curve $L_x\ni x$ whose intersection
with $D$ is supported in at most 1 point
and which satisfies,
$$M\cdot L_x\leq 2\cdot\dim (X) \dfrac{M\cdot _f C}{-(K_X+\D)\cdot _f C}$$
for every nef. divisor $M$ on $X$. In
particular by \cite{BDPP}, if 
$K_X+\Delta$ is not pseudo effective,
then $V$ is covered by $\mathbb{A}^1$'s.
\end{rmk}

This follows immediately from \ref{thm:logMM},
which we shall not repeat in this introduction
since it is designed to deal simultaneously
with a situation such as the above where one
has optimal control on the bubbling in the
boundary, but otherwise it may be supposed
to be any old rubbish, {\it i.e.} no klt, dlt,
or whatever is required, and, as we shall see,
hypothesis of much greater boundary regularity. Nevertheless,
\ref{thm:logMM}, is not a catch all. For example
the same statement is true if we suppose that
$V$ has quotient singularities with the same
proof, but the nuisance is that it does not
yield maps from $\mathbb{A}^1$ to the Vistoli
covering champ $\cV\to V$, {\it i.e.} the smallest
smooth champ with moduli $V$, \cite{V}, and, just as bad,
even in the hypothesis of \ref{rmk:interior}
with $\cX\to X$ the Vistoli covering champ
one cannot (without more thought) replace the condition of no 
$\mathbb{A}^1$'s in the higher strata
$X_{I}\backslash\D_J$ by no
$\mathbb{A}^1$'s in $\cX_{I}\backslash\D_J$.
The phenomenon which is occasioning this is
explained in \ref{rmk:15neri}, while some 
more trivial obstructions are 
recorded in \ref{ex:noAb} \& \ref{ex:campana}.
Nevertheless, modulo the usual technical
problem about smoothness assumptions in
bend \& break, one does produce $\mathbb{A}^1$'s
in the full generality of champs de Deligne-Mumford
with quasi-projective moduli. The misfortune
is that one cannot necessarily guarantee
without conditions such as \ref{rmk:interior},
and even if there were no boundary, that
these curves pass through a specified point.
Consequently, even for klt. surfaces, 
there
are things that are being missed.  
For example if $S$ has no boundary, and an
ample anti-canonical divisor, 
\ref{thm:logMM} might 
only yield a single map from $\mathbb{P}^1\to \mathcal{S}$ to
the Vistoli covering champ, even though its
known,  \cite{KMcK}, that $\mathcal{S}$ is ruled.
Similarly, \cite[2.1.3]{blochp}, bubbles (strong Gromov sense) cannot form in
the boundary of minimal quasi-projective surfaces,
but, as we've said, by \ref{rmk:15neri}, this is
necessarily a global phenomenon, and not something local that
can be attributed to the absence of parabolic champs in
the strata of the boundary.
On the other hand, these aren't obstructions
to proving a cone theorem, which, \ref{thm:cone},  we duly do, 
to wit:
 
\begin{rmk}\label{rmk:cone}
Let $X$ be a projective algebraic variety
over an algebraically closed field of any
characteristic and let 
$\cX\to X$ be a smooth tame champ de Deligne-Mumford over it
({\it i.e.} has the same moduli),
with $\D=\sum_i\D_i$  a snc
divisor on $\cX$, then
there exists a countable family $\{L_k\subset X\}$ 
of curves whose induced champs $\cL_k\to L_k$
are parabolic ({\it i.e.} admit maps 
from some $\mathbb{A}^1$) in some 
stratum $\cX_{I_k}\backslash\D_{J_k}$, cf. \ref{conv:X_I}, satisfying,
$$
 0<-(K_{\cX}+\D)\cdot \cL_k\leq  2\cdot \dim(X),
$$
such that 
\begin{equation}
\overline {\mathrm{NE}}(\cX)=
 \overline {\mathrm{NE}}(\cX)_{(K_\cX+\D)\geq 0} + \sum_k \R_{+} [\cL_k]
\end{equation}
and the parabolic rays $\R_{+}[\cL_k]= \R_{+} [L_k]$ 
are locally discrete in $\mathrm{N}_1(\cX)_{(K_\cX+\D)< 0}$.
\end{rmk}

The bound $2\cdot\dim (X)$ is, of course, not
optimal, and while it's easy enough to get
$\dim (X)+1$ instead, one has to vary the
construction, albeit even with the construction
as is one quickly gets better than $\dim (X)+1$
if the situation is very non-schematic, cf.
\ref{rmk:bounds}. A final technical remark 
is that we never do any deformation theory
which is more complicated than that from an
honest curve to a champ. This eliminates the
habitual difficulties such as the graph of
a map may not be an embedding, and one doesn't
even need to know that there is such a thing
as a Hilbert champ of a champ, even though,
\cite{HR}, it's re-assuring that there is.
Nevertheless one does need to know that every
smooth tame champ de Deligne-Mumford $\cC\to C$
over a curve admits a map from a curve.
One could have gone through the same hoops
as \cite[Expos\'e X]{SGA1} to deduce this
in all characteristics from \cite{BN}, but
it seems easier to just prove it directly,
\ref{prop:courbes}, and, of course, one
gets a description of the tame fundamental
group of $\cC$ en passant.



%
%
\section{Preliminaries}\label{S:prel}
%

%
\subsection{Set-up}\label{SS:set-up}
%

We work over an algebraically closed filed $k$ of any characteristic.
For brevity we make, 

\begin{defn}\label{defn:champ}
By a champ $\cX$ is to be understood a tame champ de Deligne-Mumford over $k$,
with quasi-projective moduli $\pi:\cX\to X$. As such $\cX/k$ is separated,
and it is proper iff $X$ is projective. Thus, consistent with the ideas in
\cite[expos\'e VI]{SGA1}, 
the above data may also, as convenience of exposition dictates, be referred to
as a champ over $X$,
and
the mis-translation stack will be eschewed.
\end{defn}

Let $\cX$ be a champ. 
A Cartier divisor $\D=\sum_i\D_i$ on $\cX$ 
is said to be {\it snc},
if $\cX$ is smooth at every
point of the support,
every irreducible component $\D_i$ of $\D$ is
smooth, and the components through any geometric
point form a system of parameters.
For any effective
Weil divisor with $\D=\sum_i\D_i$
its expression as a sum of
irreducible components, and
$I$ a set of irreducible components of 
$\D$ we define the I-th stratum:
\begin{equation}\label{eq:X_I}
\cX_I:=\bigcap_{i\in I} \D_i.
\end{equation}
When it has
sense, {\it e.g.} $\cX$ CM,
each $\D_i$ Cartier with local equations
forming a regular sequence, and $\cX_I$
reduced, and only when it has sense as a $\Q$-Cartier divisor, 
its canonical divisor is:
\begin{equation}\label{eq:K_X_I}
K_{\cX_I}:=(K_\cX+\sum_{i\in I}\D_i)_{|\cX_I}.
\end{equation}
and we continue to understand $K_{\cX_I}$
by this formula,
and say $K_{\cX_I}$ is $\mathbb{Q}$-Cartier
even if this only has
the sense that the above right hand side
is the restriction of a $\mathbb{Q}$-Cartier
divisor on $\cX$.
\begin{convention}\label{conv:X_I}
On writing $f:C\to \cX_I$, 
or $\cC\subset \cX_I$,
unless otherwise specified, we suppose that $I$ is the maximal set of components of $\D$ containing $f(C)$, respectively $\cC$.
In the implicit presence of such an $I$, we 
denote by $J$ the set of components complementary to $I$ and
should every $\D_j$, $j\in J$, be $\mathbb{Q}$-Cartier we put:
$$
\D_J:=\sum_{j\in J}{\D_j}_{|\cX_I}.
$$
\end{convention}

%
\subsection{Curves and Uniformisation}\label{ss:DMrat}
%

\begin{defn}\label{defn:euler}

Let $\cC_0\to C$ be a (smooth connected) champ over a curve, without
generic stabiliser, then we extend the topological
Euler characteristic $C\mapsto \chi(C)$ by way of,
$$\chi(\cC_0):=\chi(C) + \sum_{c\in C} (\frac{1}{n_c}-1)$$
where
$n_c$ the order of the local monodromy. For a
general (smooth connected) champ $\cC\to C$, with generic
stabiliser $G$, there is a 
fibration $\cC\to\cC_0$,
$\cC_0$ as above,
with fibre $B_G$ and:
$$ \vert G\vert  \chi(\cC) := \chi(\cC_0)$$
Should the champ be proper,
this is equally the negative of the degree of the
canonical bundle, and should $\chi(\cC)>0$ we will
say that the champ is parabolic.

\end{defn}

The presence of stabilisers, particularly generic
ones, is a recipe for technical problems, {\it e.g.}
graphs of mappings may fail to be embeddings,
which we eliminate by way of,

\begin{prop}\label{prop:courbes}
Let $\pi:\cC\to C$ be a (smooth connected) champ over a curve, then
there is a finite proper map from a curve $B$ to
$\cC$. Better still,
\begin{enumerate}
\item[(a)] $\chi(B)<0$, respectively, $=0$, or $>0$,
whenever the same is true of $\chi(\cC)$.
\item[(b)] If $\chi(\cC)\leq 0$, then we may 
take $B\to\cC$ to be \'etale, and even realise
$\cC$ as $[B/E]$ for $E$ an extension,
$$
\begin{CD}
1@>>> G@>>> E@>{\rho}>> H@>>> 1
\end{CD}
$$
acting via $\rho$ for $H$ a sub-group of $\mathrm{Aut}(B)$,
and $G$ the monodromy around a generic point of $\cC$.
\item[(c)] Otherwise $B=\mathbb{P}^1$ or $\mathbb{A}^1$.
\end{enumerate}
\end{prop} 
\begin{proof}
The assertion is well known if $G$ is trivial,
{\it e.g.} \cite[Lemma 6.5]{Kbug}.
In particular by way of a base change, 
$F\to\cC_0$, \'etale should (b) apply, 
with $\cC_0$ as per \ref{defn:euler},
and $F$ a curve we may suppose  that $C=\cC_0$.
Now write $\cC=[U/R]$ for some \'etale cover
$U\to\cC$. By hypothesis $(s,t):R\to U\times U$ is
proper, and taking $U$ sufficiently fine,
we may suppose that its image is $U\times_C U$.
In addition, the  stabiliser,
$S:=(s,t)^{-1}(\D)$, $\D\hookrightarrow U\times_C U$
the diagonal, admits a, non-canonical, isomorphism
$\phi:S\to U\times G$,
while an arbitrary fibre of $R\to U\times_C U$
may be identified with the stabiliser of its
source, so $R\to U\times_C U$ is \'etale, whence
refining $U$ as necessary we may suppose that
it admits a section $\sigma$. Conjugation
by $\sigma$ combined with $\phi$, yields
a map $\Sigma:U\times_C U\to \mathrm{Aut}(G)$,
whose image in $\mathrm{Out}(G)$ is a co-cycle,
thus affording,
$$\bar{\Sigma}\in\mathrm{H}_{\mathrm{\acute{e}t}}^1(C, \mathrm{Out}(G))$$
The failure of $\Sigma$ to be a co-cycle
in $\mathrm{Aut{G}}$ is rather precise,
{\it viz:} applying the C\v{e}ch co-boundary operator
at the level of arrows yields a map,
$g:V\times_C V\times_C V\to G$ such
that the co-boundary of $\Sigma$ is
the inner automorphism associated to
$g$. The naturality of $g$ implies, in
the usual C\v{e}ch notation, that:
$$g_{\alpha\beta\delta}^{-1}g_{\alpha\gamma\delta}g_{\alpha\beta\gamma}
=\Sigma_{\alpha\beta}(g_{\beta\gamma\delta}).$$

At this point we must distinguish the
case (b) and (c) of the proposition.
In the former $C$ is a (tame) \'etale $K(\pi, 1)$,
and the above condition on $g$ says exactly
that there is a group extension,
$$1\to G\to E\to \pi \to 1$$
with action of $\pi$ by outer automorphisms
defined via $\bar{\Sigma}\in \mathrm{Hom}(\pi, \mathrm{Out}(G))$,
and implied 2 co-cycle defined by $g$
viewed as a map from $\pi\times\pi$ to $G$.
Here $\pi$ is the 
(prime to the characteristic)
pro-fundamental
group, so this proves the assertion
(b) when $C=\cC_0$ which is all that
we'll need in the sequel, while in
general $\cC_0$ is also a (tame) \'etale
$K(\pi, 1)$, and the same argument
works at the price of some notational
complication occasioned on replacing
$U\times_C U$ by the arrows $R_0$
defining $\cC_0$, and so forth. 

As to
case (c) of the proposition, we may
suppose that $\cC_0=C$ is either
$\mathbb{P}^1$ or $\mathbb{A}^1$, so
that in either case $\bar{\Sigma}$ is 
homologous to zero. Adjusting both
$\phi$, and $\sigma$ accordingly,
we may suppose that $\Sigma$ is
trivial, and $g$ takes values in
the centre $Z$ of $G$. In particular
by the above formula it defines a
class,
$$\bar{g}\in\mathrm{H}^2_{\mathrm{\acute{e}t}}(C,Z)$$
necessarily null if $C=\mathbb{A}^1$,
and otherwise null after a base change
of the form,
$$\mathbb{P}^1\to\mathbb{P}^1 : x\mapsto x^l$$
for some sufficiently large $l$.
Consequently after a base change, and
an adjustment of $\sigma$ by central
elements, we obtain a 1-isomorphism,
given on arrows by,
$$R\to C\times_C G: f\mapsto (s(f), f\sigma(s\times t(f))^{-1})$$
for $G$ acting trivially on $C$, {\it i.e.}
$\cC=C\times B_G$, and we can take
$B=C$ mapping to $\cC$ via the 
natural projection $\mathrm{pt}\to B_G$.
\end{proof}

To which let us add,

\begin{rmk}\label{rmk:curves}
{\em Of course $\cC_0\to C$ is parabolic only if the
number of non-schematic points is at most 3, and
should it be 3 then the universal cover of
$\cC_0$ is $\mathbb{P}^1$, and one may enumerate
the possibilities where this occurs. In the
general case, however, there is no such simple
enumeration, so that being more precise than
\ref{prop:courbes}(c) is pointless. Nevertheless
it gives an equivalent criteria for parabolicty,
{\it viz:} A champ is parabolic iff it is dominated
by $\mathbb{A}^1$, or $\mathbb{P}^1$
should it be proper.}
\end{rmk}

The above suggests that a useful
way to think of a quasi-projective
variety is as a champs with infinite
monodromy around the boundary. Unfortunately,
this is, a priori, technically vacuous,
and so we make,

\begin{defn}\label{defn:pullback}
A smooth quasi projective curve is a smooth projective
curve $C$ together with a reduced divisor $D$. Exactly
as per \ref{defn:euler} the topological Euler characteristic
of $C\backslash D$ is equally the negative of the
log-canonical degree $K_C+D$. 
For $\D=\sum_{i}\D_i$ an
effective Weil divisor on a champ
$\cX$, and
$f:C\to\cX_I$ respecting the
convention \ref{conv:X_I}, the induced
quasi-projective curve is $C$ together
with the unique reduced divisor whose
support coincides with $f^*\D_J$. 
\end{defn}

%
\subsection{Dimension counts}\label{SS:dimcounts}
%


We put ourselves in the situation
of \ref{conv:X_I}, including, by the way,
a possibly empty set of indices $I$. 
We wish to study $\Mor(C,\cX_I)$.
As it happens,
the Hilbert champ of a champ is known to exist, \cite{HR},
so the aforesaid space of morphisms may viewed as
a sous-champ of a connected component of the graph.
This is not, however, in anyway necessary since
$f^*\cP^{\infty}_{\cX_I}$ is a sheaf of admissible
$\cO_C$ algebras,
\cite{formal},
whose formal spectrum is the
completion of $C\times\cX_I$ in the graph. As
such the local theory is wholly (formal) scheme like,
and one could just as well define $\Mor(C,\cX_I)$
as an open subscheme of the Chow scheme of the
moduli. In any case, this object is known to
exist, and thinking of a quasi-projective curve
as a champ with infinite monodromy on the boundary,
one realises that the relevant object of study is,

\begin{defn} Let things be as above, then 
the sous-champ,
$$\Mor(C,\cX_I,\D_J)\subset \Mor(C,\cX_I)$$ is the parameter space of morphisms $h:C\to \cX_I$ such that $h^*(\D_J)=f^*(\D_J)$. 
In addition, bearing in mind both equation (\ref{eq:K_X_I}) and 
the convention \ref{conv:X_I}, we say, cf. \cite[Definition 3.1]{Kbug},
\begin{enumerate}
\item[(i)] There are enough deformations of $f$ in $\cX_I$ if,
$$
 \dim_{[f]} \Mor(C,\cX_I)\geq 
 -K_{\cX_I}\cdot_{f}C +(1-g(C))\dim(\cX_I).
$$
and there are enough deformations in $\cX_I$ if this holds
for all $f:C\to\cX_I$.
\item[(ii)] There are enough deformations
of  $f$ in $\cX_I\backslash\D_J$ if,
$$
 \dim_{[f]} \Mor(C,\cX_I,\D_J)\geq 
 -( K_{\cX}+\D)\cdot_{f}C +(1-g(C))\dim(\cX_I).
$$
and there are enough deformations in  $\cX_I\backslash\D_J$ if this holds
for all $f:C\to\cX_I$.
\end{enumerate}
\end{defn}



As in the compact case, we will need to consider curves passing through a fixed point.
\begin{defn}
Conventions enforce, let $f:C\to \cX_I$ be a curve and $\Gamma\subset C$ a finite subset such that $\Gamma\cap f^*\D_J=\emptyset$. The sous-champ
$\Mor(C,\cX_I,\D_J, f_{|\Gamma})\subset \Mor(C,\cX_I,\D_J)$ is the parameter space of morphisms $h:C\to \cX_I$ which furthermore verify $h_{|\Gamma}=f_{|\Gamma}$.
Manifestly if $f(\Gamma)$ is contained in the smooth locus
of $\cX_I$ then,
\begin{equation}\label{eq:gamma}
\dim_{[f]} \Mor(C,\cX_I,\D_J, f_{|\Gamma})\geq \dim_{[f]} \Mor(C,\cX_I,\D_J) 
-\vert\Gamma\vert \dim(\cX_I)
\end{equation}
\end{defn}

To which, we have the following minor variant of Mori's estimate,

\begin{lem}\label{lem:dim2}
Let things be as above then,
\begin{enumerate}
\item[(i)] If $\cX_I$ is LCI, and $f$ meets the smooth
locus of the same, then there are enough deformations
of $f$ in $\cX_I$. 
\item[(ii)] If there are enough deformations in
$\cX_I$ and $\D_J$ is Cartier where it meets $f$,
then there are enough deformations in $\cX_I\backslash\D_J$.
\end{enumerate}
In particular if  $\cX$ is smooth, then there are
always enough deformations in $\cX\backslash\D$, and if
$(\cX,\D)$ is smooth
with snc boundary then there are even enough deformations
in $\cX_I\backslash\D_J$ for every stratum.
\end{lem}
\begin{proof} Case (i) follows for the same reason as
\cite[Propostion 3]{MoriOriginal}, {\it i.e.} the
functoriality of the obstruction group \cite[expos\'e III, 5.1]{SGA1},
and the same calculation as \cite[II.1.3]{K}.
Case (ii) follows from case (i), since
fixing the intersection with the boundary is at most $\D_J\cdot_{f}C$ conditions.
\end{proof}

%
\section{Looking for rational curves}
%
%
\subsection{An instructive case}\label{SS:easy}
%

As ever we place ourselves in the
set up
\ref{defn:champ},
with $F_b:C\to\cX_I$ 
a one dimensional
family of morphisms
respecting the
convention \ref{conv:X_I},
{\it i.e.} we have a morphism,
$$
F:C\times B\to \cX_I
$$
where $B$ is a (not necessarily complete)
smooth curve. Suppose further that every
$\D_j$, $j\in J$, is $\mathbb{Q}$-Cartier, and
for some $c\in C$,
\begin{enumerate}
\item[(i)] $F(c\times B)$ is a point and $c\notin F_b^*(\D_J)$;
\item[(ii)] for any $b\in B$ the pull-back of the boundary $F_b^*(\D_J)$ is constant on $C=C\times b$;
\item[(iii)] $F$ is generically finite.
\end{enumerate}
Let $\overline B$ be a smooth compactification of $B$. We denote by $f$ the 
composition $\pi F$,
as well as the induced rational map:
$$
f:C\times \overline B \dashrightarrow X_I.
$$
By (i) this cannot be defined on all of
$c\times\overline B$.
Let $X$ be projective and $S$ a smooth surface 
obtained by a sequence of blow ups in closed
points such that the rational map
 $\overline{f}:S\to X_I$ is a morphism, and: 
\begin{enumerate}
\item[(iv)] $\overline{f}$ is
a smooth and minimal, i.e. it does not contract any $(-1)$-curve.
\end{enumerate} 

Let $E=\sum_{\alpha} e_{\alpha} E_{\alpha}$ be the exceptional divisor in $S$
over some indeterminancy in $c\times\overline B$. 
Notice that by (ii) the divisor $\overline{f}^*\D_J$ is supported either on 
$(\overline{f_0}^*\D_J)\times \overline{B}$, 
some fixed $0\in B$,
or on the components of the exceptional divisor. Therefore any component $E_a$ 
of $E$ such that $\overline{f}(E_{\alpha})\not\subset\D_J$ can only meet $\D_J$ in another component 
$E_{\beta}$ such that $\overline{f}(E_{\beta})\subset \D$ or along the proper transform 
of the fibre, 
$C\times\{b\}$, 
through the indeterminacy should
this be contained in $\overline{f}^*\D$. 
Summarising for strata, yields:

\begin{lem}\label{lem:int} 
Let everything be as above, with $E_\alpha$  a component  of the exceptional 
divisor of  $\overline{f}$. Suppose $\overline{f}(E_\alpha)\subset X_{I'}$, 
with $I'\supset I$ maximal. Then $E_\alpha$ can only meet 
$\D_j,\ j\notin I',$ in another component 
$E_\beta$ or along the proper transform $\tilde C_b$
of 
a fibre, $C\times\{b\}$, through an indeterminacy,
and this only if $\overline{f}(E_\beta)\subset \D_{J'}$ or 
$\overline{f}(\tilde C_b)\subset \D_{J'}$.
\end{lem}

We will consider the dual graph associated to 
$E$ together with
the proper transform $\tilde C$ of $C\times b$,  
where the graph shall be rooted, and we endow it with the metric 
given by the distance from the same. Observe that a vertex at maximal distance from the root 
has valency one.

\begin{lem}\label{lem:valency1} 
Let $v$ be a vertex having valency one, and denote by $E_v$ the corresponding component of the exceptional divisor. Suppose that 
\begin{equation}
{\textrm{ $E_v$ is not contracted by $\overline{f}$ to a point.}}
\end{equation}
Let $I'$ be the maximal set such that $\overline{F}(E_v)\subset X_{I'}.$
Then $\overline{f}(E_v)$ yields a 
map from $\mathbb{A}^1$ to
$ X_{I'}\backslash \D_{J'}$.
\end{lem}
\begin{proof}
Since the vertex has valency one, by Lemma \ref{lem:int} the component $E_v$ can meet $(\overline f)^{-1}\D_{J'}$ at most at the point corresponding to the edge which connects it to the rest of the tree. Therefore it yields 
at worst a map from $\mathbb{A}^1$ to
$X_{I'}\backslash\D_{J'}$.
\end{proof}

In light of Lemma \ref{lem:valency1} the natural path to follow is to contract 
the vertices of valency one whose image in $X$ is a point and analyse what happens on the new (singular) surface. The morphism  $F:C\times B\to \cX_I$ will then be used to obtain 
a parabolic champ inside $\cX_I$.
This will be undertaken in the following subsection and lead to Proposition \ref{prop:2}.

%
\subsection{The key algorithms}\label{SS:alg}
%
We keep the notation of the previous section \S \ref{SS:easy}. We set $S_0:=S$. When no confusion is possible, by abuse of language, we will often identify a vertex to the corresponding irreducible component and speak, for instance, of contraction or image of a vertex. 

We construct inductively two families of surfaces according to the following algorithms: 

\begin{alg}[Strong rule]\label{alg:strongrule}
The surface $S_{n+1}$ is constructed by contracting all the vertices of valency one in $S_n$ whose proper transform in $S$ is not a $(-1)$-curve.
\end{alg}

\begin{alg}[Weak rule]\label{alg:weakrule}
The surface $W_{n+1}$, $W_0=S_0$, is constructed by contracting all the vertices of valency one in $W_n$ whose image in $X$ is a point.
\end{alg}

\begin{rmk}\label{rmk:minimality}
 {\rm By the minimality of the resolution no vertex of the graph whose corresponding component is contracted by $\overline f$ can be a $(-1)$-curve.\hfill$\square$}
\end{rmk}

\begin{prop}\label{prop:1}
Let everything be as in \S \ref{SS:easy}.
Consider the dual graph associated to the exceptional divisor of $S\to C\times \overline B$, union proper transform $\tilde C$ of $C\times b$,  rooted at  $\tilde C$, and endowed with the metric given by the distance from the root.
Let $S_n$ be a surface obtained via the algorithm \ref{alg:strongrule}. Then
\begin{enumerate}
\item[(i)] the induced graph  on $S_n$ is a tree;
\item[(ii)] any one-dimensional fibre of $S\to S_n$ is a chain;
\item[(iii)] any vertex 
in $S_n$ viewed as a vertex in the graph on $S$ meets at most one contracted connected component.
\item[(iv)] $S_n$ has at worst quotient singularities.
\end{enumerate}
\end{prop}

\begin{proof}[Proof of Proposition \ref{prop:1}]
Items (i) \& (iv) follows from item (ii) and from the type of algorithm we are using. 
Item (ii) follows from item (iii).

We are then reduced to prove item (iii). 
Let $v$ be a vertex, meeting $s$ 
contracted connected component. 
In particular viewed as a vertex in $S$ it has valency
at least $s+1$, {\it i.e.} $s$ edges for the contracted
components, and one giving the unique path to the root.
Now consider undoing the procedure whereby $S$ was
obtained from $C\times \overline B$, {\it i.e.} contract $-1$
curves in the reverse order. By construction this
never destroys any of the above edges in the induced
graph, but our vertex of interest eventually becomes
a $-1$-curve, so, $s+1\leq 2$.
\end{proof}

\begin{defn}\label{defn:W}
{\em By Remark \ref{rmk:minimality} the previous proposition also holds for $W_n$,
and we let $W$ be the conclusion of the algorithm.}
\end{defn}

Now denote by $\Gamma$ the graph of $F$
in $C\times\bar{B}\times\cX_I$, and $\vert \Gamma\vert$
its moduli, equivalently the graph of $f$, and
consider
the following diagram,
\begin{equation}\label{eq:graph} 
\xymatrix{
\cW\ar[r]\ar[d]&\Gamma\ar[d]\\
W\ar[r]&\vert\Gamma\vert
}
\end{equation}
where $\cW$ is the normalisation of the dominant component.
It is tame because
$\cX_I$ is, and an isomorphism over $C\times B$. Better still:
\begin{lem}\label{lem:purity}
Let $\cV\to W$ be the Vistoli covering champ of $W$,
then there is a smooth champ $\tilde\cW\to \cW$
over $W$ with trivial generic stabiliser such that:
$\tilde\cW\to\cV$ is the extraction of roots
of
components (irrespective of the order) of
a (possibly empty) snc divisor in $\cV$,
{\it i.e.} everywhere \'etale locally
$\cO_{\tilde\cW}=\cO_{\cV}[z]/(z^l=x)$, 
or $\cO_{\tilde\cW}=\cO_{\cV}[z,w]/(z^l=x,w^m=y)$
for
$x=0$, repectively $xy=0$, an equation of the 
reduced fibre over B, and $l,m\in\mathbb{N}$
prime to the characteristic and depending on
the component.
\end{lem}
\begin{proof}
By \ref{prop:1} (iv), the Vistoli covering champ
$\cV\to W$ exists, and the pre-image of any fibre
over $B$ is a snc divisor. Both $\cV$ and $\cW$
contain $C\times B$ as an open set, so letting
$\tilde\cV$ be the normalisation of the dominant
component of $\cW\times_W\cV$, the champ
$\tilde\cV$ has pure ramification over a 
simple normal crossing divisor. The tame 
fundamental group of a simple normal
crossing divisor is what one expects,
\cite[Exp. XIII, Cor. 5.3]{SGA1}, so,
in particular away from a crossing of components
the extraction of a root as above for some
$l_k$ depending on the component $k$, and
taking $l_k$th roots in each component
individually, irrespective of any order,
yields $\tilde\cW$.
\end{proof}

The promised generalisation of Lemma \ref{lem:valency1} may now be given. 

\begin{prop}\label{prop:2} 
Every irreducible component of the exceptional divisor in $\tilde\cW$ corresponding to the vertices of valency one in $W$ yields a parabolic champ in $\cX_{I'}$, for some $I'\supseteq I$.
\end{prop}

\begin{proof}
Let $v$ be a valency one vertex in $W$,
$\cE_v$ the corresponding irreducible 
component of the exceptional divisor in $\tilde\cW$,
with
$\tilde F$ the morphism from $\tilde\cW$ to $\cX$, and $I'\supseteq I$ 
maximal amongst sets
such that  $\tilde F(\cE_v)\subset \cX_{I'}$.
By Proposition \ref{prop:1}, the component $\cE_v$ meets 
the rest, $\cR$, of the fibre of $\tilde \cW$ over $B$
in at most one point. By \ref{lem:purity}, the picture
is as follows:


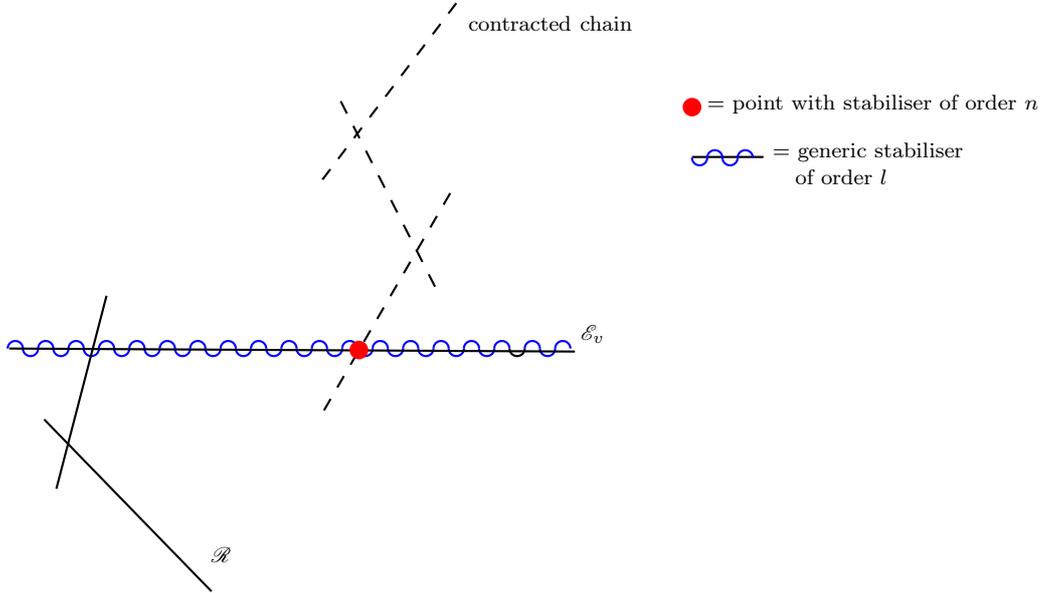
\psset{xunit=1.0cm,yunit=1.0cm,algebraic=true,dotstyle=o,dotsize=3pt 0,linewidth=0.8pt,arrowsize=3pt 2,arrowinset=0.25}
\begin{pspicture*}(-4.3,-2.66)(11.7,5.8)

\parametricplot[linecolor=blue]{0.0}{3.141592653589793}{1*0.1*cos(t)+0*0.1*sin(t)+-3.4|0*0.1*cos(t)+1*0.1*sin(t)+1.22}

\parametricplot[linecolor=blue]{3.141592653589793}{6.283185307179586}{1*0.1*cos(t)+0*0.1*sin(t)+-3.2|0*0.1*cos(t)+1*0.1*sin(t)+1.22}

\parametricplot[linecolor=blue]{0.0}{3.141592653589793}{1*0.1*cos(t)+0*0.1*sin(t)+-3|0*0.1*cos(t)+1*0.1*sin(t)+1.22}

\parametricplot[linecolor=blue]{3.141592653589793}{6.283185307179586}{1*0.1*cos(t)+0*0.1*sin(t)+-2.8|0*0.1*cos(t)+1*0.1*sin(t)+1.22}

\parametricplot[linecolor=blue]{0.0}{3.141592653589793}{1*0.1*cos(t)+0*0.1*sin(t)+-2.6|0*0.1*cos(t)+1*0.1*sin(t)+1.22}

\parametricplot[linecolor=blue]{3.141592653589793}{6.283185307179586}{1*0.1*cos(t)+0*0.1*sin(t)+-2.4|0*0.1*cos(t)+1*0.1*sin(t)+1.22}

\parametricplot[linecolor=blue]{0.0}{3.141592653589793}{1*0.1*cos(t)+0*0.1*sin(t)+-2.2|0*0.1*cos(t)+1*0.1*sin(t)+1.22}

\parametricplot[linecolor=blue]{3.141592653589793}{6.283185307179586}{1*0.1*cos(t)+0*0.1*sin(t)+-2|0*0.1*cos(t)+1*0.1*sin(t)+1.22}

\parametricplot[linecolor=blue]{0.0}{3.141592653589793}{1*0.1*cos(t)+0*0.1*sin(t)+-1.8|0*0.1*cos(t)+1*0.1*sin(t)+1.22}

\parametricplot[linecolor=blue]{3.141592653589793}{6.283185307179586}{1*0.1*cos(t)+0*0.1*sin(t)+-1.6|0*0.1*cos(t)+1*0.1*sin(t)+1.22}

\parametricplot[linecolor=blue]{0.0}{3.141592653589793}{1*0.1*cos(t)+0*0.1*sin(t)+-1.4|0*0.1*cos(t)+1*0.1*sin(t)+1.22}

\parametricplot[linecolor=blue]{3.141592653589793}{6.283185307179586}{1*0.1*cos(t)+0*0.1*sin(t)+-1.2|0*0.1*cos(t)+1*0.1*sin(t)+1.22}

\parametricplot[linecolor=blue]{0.0}{3.141592653589793}{1*0.1*cos(t)+0*0.1*sin(t)+-1.|0*0.1*cos(t)+1*0.1*sin(t)+1.22}

\parametricplot[linecolor=blue]{3.141592653589793}{6.283185307179586}{1*0.1*cos(t)+0*0.1*sin(t)+-0.8|0*0.1*cos(t)+1*0.1*sin(t)+1.22}

\parametricplot[linecolor=blue]{0.0}{3.141592653589793}{1*0.1*cos(t)+0*0.1*sin(t)+-0.6|0*0.1*cos(t)+1*0.1*sin(t)+1.22}

\parametricplot[linecolor=blue]{3.141592653589793}{6.283185307179586}{1*0.1*cos(t)+0*0.1*sin(t)+-0.4|0*0.1*cos(t)+1*0.1*sin(t)+1.22}

\parametricplot[linecolor=blue]{0.0}{3.141592653589793}{1*0.1*cos(t)+0*0.1*sin(t)+-0.2|0*0.1*cos(t)+1*0.1*sin(t)+1.22}

\parametricplot[linecolor=blue]{3.141592653589793}{6.283185307179586}{1*0.1*cos(t)+0*0.1*sin(t)+-0.0|0*0.1*cos(t)+1*0.1*sin(t)+1.22}

\parametricplot[linecolor=blue]{0.0}{3.141592653589793}{1*0.1*cos(t)+0*0.1*sin(t)+0.2|0*0.1*cos(t)+1*0.1*sin(t)+1.22}

\parametricplot[linecolor=blue]{3.141592653589793}{6.283185307179586}{1*0.1*cos(t)+0*0.1*sin(t)+0.4|0*0.1*cos(t)+1*0.1*sin(t)+1.22}

\parametricplot[linecolor=blue]{0.0}{3.141592653589793}{1*0.1*cos(t)+0*0.1*sin(t)+0.6|0*0.1*cos(t)+1*0.1*sin(t)+1.22}

\parametricplot[linecolor=blue]{3.141592653589793}{6.283185307179586}{1*0.1*cos(t)+0*0.1*sin(t)+0.8|0*0.1*cos(t)+1*0.1*sin(t)+1.22}

\parametricplot[linecolor=blue]{0.0}{3.141592653589793}{1*0.1*cos(t)+0*0.1*sin(t)+1|0*0.1*cos(t)+1*0.1*sin(t)+1.22}

\parametricplot[linecolor=blue]{3.141592653589793}{6.283185307179586}{1*0.1*cos(t)+0*0.1*sin(t)+1.2|0*0.1*cos(t)+1*0.1*sin(t)+1.22}

\parametricplot[linecolor=blue]{0.0}{3.141592653589793}{1*0.1*cos(t)+0*0.1*sin(t)+1.4|0*0.1*cos(t)+1*0.1*sin(t)+1.22}

\parametricplot[linecolor=blue]{3.141592653589793}{6.283185307179586}{1*0.1*cos(t)+0*0.1*sin(t)+1.6|0*0.1*cos(t)+1*0.1*sin(t)+1.22}

\parametricplot[linecolor=blue]{0.0}{3.141592653589793}{1*0.1*cos(t)+0*0.1*sin(t)+1.8|0*0.1*cos(t)+1*0.1*sin(t)+1.22}

\parametricplot[linecolor=blue]{3.141592653589793}{6.283185307179586}{1*0.1*cos(t)+0*0.1*sin(t)+2|0*0.1*cos(t)+1*0.1*sin(t)+1.22}

\parametricplot[linecolor=blue]{0.0}{3.141592653589793}{1*0.1*cos(t)+0*0.1*sin(t)+2.2|0*0.1*cos(t)+1*0.1*sin(t)+1.22}

\parametricplot[linecolor=blue]{3.141592653589793}{6.283185307179586}{1*0.1*cos(t)+0*0.1*sin(t)+2.4|0*0.1*cos(t)+1*0.1*sin(t)+1.22}

\parametricplot[linecolor=blue]{0.0}{3.141592653589793}{1*0.1*cos(t)+0*0.1*sin(t)+2.6|0*0.1*cos(t)+1*0.1*sin(t)+1.22}

\parametricplot[linecolor=blue]{3.141592653589793}{6.283185307179586}{1*0.1*cos(t)+0*0.1*sin(t)+2.8|0*0.1*cos(t)+1*0.1*sin(t)+1.22}

\parametricplot[linecolor=blue]{0.0}{3.141592653589793}{1*0.1*cos(t)+0*0.1*sin(t)+3|0*0.1*cos(t)+1*0.1*sin(t)+1.22}
\parametricplot{3.141592653589793}{6.283185307179586}{1*0.1*cos(t)+0*0.1*sin(t)+3.2|0*0.1*cos(t)+1*0.1*sin(t)+1.22}

\parametricplot[linecolor=blue]{0.0}{3.141592653589793}{1*0.1*cos(t)+0*0.1*sin(t)+3.4|0*0.1*cos(t)+1*0.1*sin(t)+1.22}

\parametricplot[linecolor=blue]{3.141592653589793}{6.283185307179586}{1*0.1*cos(t)+0*0.1*sin(t)+3.6|0*0.1*cos(t)+1*0.1*sin(t)+1.22}

\parametricplot[linecolor=blue]{0.0}{3.141592653589793}{1*0.1*cos(t)+0*0.1*sin(t)+3.8|0*0.1*cos(t)+1*0.1*sin(t)+1.22}

\psline(-3.48,1.22)(3.96,1.18)
\psline[linestyle=dashed,dash=5pt 5pt](0.66,0.4)(2.32,3.28)
\psline[linestyle=dashed,dash=5pt 5pt](2.12,2.04)(0.88,4.5)
\psline[linestyle=dashed,dash=5pt 5pt](0.64,3.46)(2.4,5.8)
\psline(-2.2,1.92)(-2.86,-0.64)
\psline(-3.02,0.28)(-0.82,-2)
\psline(5.5,3.76)(6.44,3.76)
\begin{scriptsize}
\rput[bl](4.04,1.3){{$\mathcal E_v$}}

\rput[bl](2.56,5.44){{contracted chain}}
\psdots[dotsize=7pt 0,dotstyle=*,linecolor=red](1.12,1.2)

\rput[bl](-0.82,-1.62){{$\mathcal R$}}
\psdots[dotsize=7pt 0,dotstyle=*,linecolor=red](5.5,4.42)
\rput[bl](5.7,4.34){{$=$ point with stabiliser of order $n$}}
\rput[bl](6.56,3.7){$=$ generic stabiliser}
\rput[bl](6.86,3.4){of order $l$}
\parametricplot[linecolor=blue]{3.141592653589793}{6.283185307179586}{1*0.1*cos(t)+0*0.1*sin(t)+5.6|0*0.1*cos(t)+1*0.1*sin(t)+3.75}

\parametricplot[linecolor=blue]{0.0}{3.141592653589793}{1*0.1*cos(t)+0*0.1*sin(t)+5.8|0*0.1*cos(t)+1*0.1*sin(t)+3.75}

\parametricplot[linecolor=blue]{3.141592653589793}{6.283185307179586}{1*0.1*cos(t)+0*0.1*sin(t)+6|0*0.1*cos(t)+1*0.1*sin(t)+3.75}

\parametricplot[linecolor=blue]{0.0}{3.141592653589793}{1*0.1*cos(t)+0*0.1*sin(t)+6.2|0*0.1*cos(t)+1*0.1*sin(t)+3.75}
\end{scriptsize}
\end{pspicture*}

Or, more precisely, in the notation of the proof of
\ref{lem:purity}, $\tilde\cW\to\cV$ restricted to
$\cE_v\backslash\cR$ is an \'etale covering of degree
$l^{-1}$. Therefore: 
$$
\chi(\cE_v\backslash\cR) = \frac{1}{l}\cdot (2-1-(1-\frac{1}{n}))
=\frac{1}{ln}>0,
$$
where $l$ is the order of the stabiliser of the generic point $\cE_v$, 
and $n$ is the order of the stabiliser at the non-scheme like
point on the image in $\cV$ should 
a contracted chain occur-
otherwise $\chi=1/l$, and 
the proposition is proved.
\end{proof}

%
\section{Log Bend-and-Break}
%

Let us first consider some
ways in which
Abhyankar's theorem: any positive dimensional
component of a fibre of a birational map to
a smooth variety is ruled, fails for champs,
{\it i.e.},

\begin{ex}\label{ex:noAb}{\em
Let $X$ be a smooth surface. Let $X_p$ be the blow-up of a point $p\in X$ and $E$ the exceptional divisor. Let $Y$ be the blow-up of three points $q_1, q_2, q_3$ on $E$, $F_1, F_2, F_3$ the corresponding exceptional divisors, $\tilde E$ the proper transform of $E$ in $Y$ and $f:Y\to X$ the composite morphism. 
Taking $m_i\geq 3$th roots in the $F_i$ yields a champ $\rho:\cY\to Y$ 
over $Y$, proper and birational to $X$, but
$\rho^{-1}(\tilde E)$ is not 
parabolic.}
\end{ex}
Similarly, we have the following example of Campana,
in which we find further obstructions to 
Bend-and-break in the presence of a boundary.
\begin{ex}[Campana, \cite{C2}, Example 9.19]\label{ex:campana}{\em
Take an isotrivial family of smooth plane cubics $C_t$ passing through a point $o$ degenerating to a union of 3 lines $L_1, L_2, L_3$, only one of which, say $L_1$, passes through $o$. For instance one can take $C_t=x^3+y^3+tZ^3=0$ and $o=[1:-1:0]$. Let $X$ be the blow-up of two general points on $L_1$ distinct from $o$ and let $E_1, E_2$ be the corresponding exceptional divisors. Set $\D:=E_1+E_2$. Let $\tilde o\in X$ be the point over $o$. The proper transform $\tilde L_1$
is the only rational curve in the family through $\tilde o$, but it does 
not yield 
an $\mathbb{A}^1$ in $X\backslash\D$.}
\end{ex}

These and other difficulties limit how much one can improve \ref{prop:2}, {\it viz:}

\begin{prop}\label{prop:b&b}
Let everything be as in \ref{SS:easy} and suppose
further that the family fixes not just a point but a finite set,
$\emptyset\neq\Gamma\subset C$ complimentary to $f^*\supp(\D_J)$, 
then for $H$ a nef divisor on $X$ there is a parabolic
champ $\cL$ in some possibly different stratum 
$\cX_I'\backslash\D_J'$, $I'\supseteq I$, such that:
\begin{equation}\label{eq:degree}
 H\cdot \cL \leq \frac{2(H\cdot_f C)}{|\Gamma|}.
\end{equation}
If in addition,

\noindent (a) $\cX_{I}\backslash\D_J\to X_I\backslash\vert\D_J\vert $ is \'etale. 

\noindent (b) For any $I'\supset I$, no proper sub-stratum
$X_{I'}\backslash\vert\D_{J'}\vert$ of the moduli
contains a parabolic curve.

Then the parabolic champ $\cL$ of \ref{eq:degree} may be taken
to meet $f(\Gamma)$ and lie in the stratum $\cX_{I}\backslash\D_J$.
\end{prop}

\begin{rmk}\label{rmk:sharp}{\em
Both hypotheses in Proposition \ref{prop:b&b} are necessary. 
For example,
if we remove \ref{prop:b&b} (b), 
then \ref{ex:campana} applies. The
necessity of \ref{prop:b&b} (a) 
and the further impossibility of
replacing $X_{I'}\backslash\vert\D_{J'}\vert$
by $\cX_{I'}\backslash\D_{J'}$
will be discussed in Remark \ref{rmk:15neri}.}
\end{rmk}

\begin{proof}[Proof of Proposition \ref{prop:b&b}]
Put $\Gamma=\{c_1,\ldots, c_\gamma\}$.
As a first step in the resolution of an
indeterminacy at a point on a section 
$c_i\times B$, one blows up successively
in the point where the proper transform
crosses the exceptional divisor until
the map is well defined at the proper
transform of the section. This connects
the proper transform of the fibre to
that of the section by way of a chain of
rational curves, with $E$, say, the
unique curve meeting the section.
The
degree bound \ref{eq:degree} comes from
\cite[Theorem 4]{MM}, and applies to
the total transform of such an $E$
at some such indeterminacy of $f$. Choose
this indeterminacy,
and form a graph $G$ whose vertices
are the curves in
the total transform
of $E$ together with one
other $o$ for the rest of
the curves in the fibre.
The latter intersect the
total transform in a
single point on a single
curve in the said transform,
and between the 
corresponding vertices we
add an edge, together 
with edges for all intersections
between curves in the total
transform. Finally, we root
the whole thing in $o$,
so \ref{prop:2}, or more
correctly the proof adapted
to the above graph, yields
(\ref{eq:degree}). Furthermore:

\begin{claim}\label{claim:ind}
Suppose \ref{prop:b&b} (a) \& (b), and let $v\neq o$ be a vertex in $G$ 
(or the same at any other
such indeterminacy, albeit we may not
have a degree bound there)
such that the irreducible component $E_v$ is not contracted to a point in $X$,
then $\overline f(E_v)\not\subset \vert\D_J\vert$.
Better still
each $E_v$ yields a parabolic champ in $\cX_I\backslash\D_J$.
\end{claim}
\begin{proof}[Proof of Claim \ref{claim:ind}]
By decreasing induction on the distance from the root
in the image, $K$, of the graph (more correctly dual
graph of the image) in $W$ for $W$ as per
\ref{defn:W}, albeit with no contractions
being performed whenever these occur in $o$. 
Let $v_{max}$ be a vertex at maximal distance which is not contracted in $\cX$. Since it has valency one, we may apply Proposition \ref{prop:2} to deduce that $v_{max}$ yields a 
parabolic champ in $\cX_{I'}\backslash\D_{J'}$, for some $I'\supset I$. 
By hypothesis \ref{prop:b&b} (b) we must have $I'=I$ and we are done. 
Let now $v_n$ be a vertex, corresponding to an irreducible component of the exceptional divisor, which is not contracted in $\cX$ and has distance $n$ from the root, as for example in the following picture. 

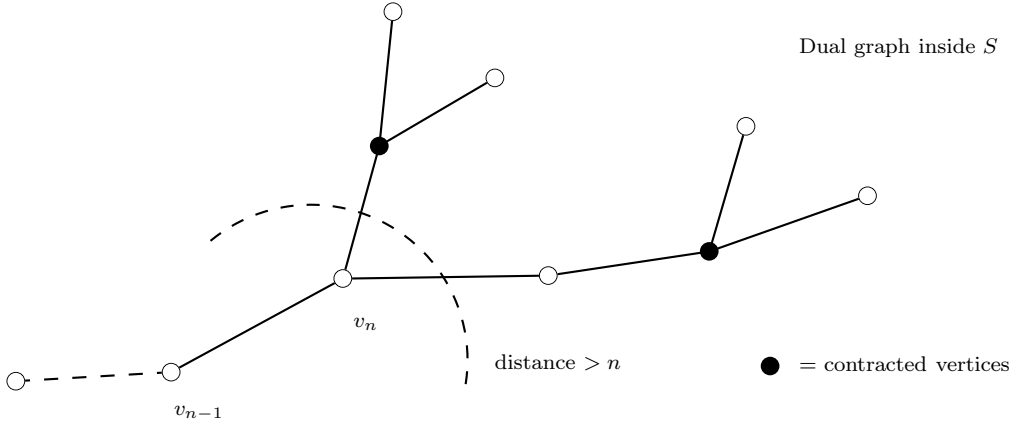
\psset{xunit=2.0cm,yunit=2.0cm}
\begin{pspicture}(-2.15,-1.12)(2.85,2.15)
\psset{xunit=1.0cm,yunit=1.0cm,algebraic=true,dotstyle=o,dotsize=3pt 0,linewidth=0.8pt,arrowsize=3pt 2,arrowinset=0.25}
\psline(-1.26,-1)(1,0.24)
\psline(1,0.24)(1.48,2)
\psline(1.48,2)(1.66,3.78)
\psline(1.48,2)(3,2.9)
\psline(1,0.24)(3.7,0.28)
\psline(3.7,0.28)(5.82,0.6)
\psline(5.82,0.6)(6.3,2.26)
\psline(5.82,0.6)(7.9,1.34)
\parametricplot[linestyle=dashed,dash=5pt 5pt]{-0.15559482797702984}{2.2667783834447546}{1*2.06*cos(t)+0*2.06*sin(t)+0.58|0*2.06*cos(t)+1*2.06*sin(t)+-0.84}
\psline[linestyle=dashed,dash=5pt 5pt](-3.3,-1.12)(-1.26,-1)
\begin{scriptsize}
\psdots[dotsize=7pt 0](-1.26,-1)
\rput[bl](-1.22,-1.64){$v_{n-1}$}
\psdots[dotsize=7pt 0](1,0.24)
\rput[bl](1.14,-0.44){$v_n$}
\psdots[dotsize=7pt 0](-3.3,-1.12)
\psdots[dotsize=7pt 0,dotstyle=*](1.48,2)
\psdots[dotsize=7pt 0](3.7,0.28)
\psdots[dotsize=7pt 0](1.66,3.78)
\psdots[dotsize=7pt 0](3,2.9)
\psdots[dotsize=7pt 0,dotstyle=*](5.82,0.6)
\psdots[dotsize=7pt 0](7.9,1.34)
\psdots[dotsize=7pt 0](6.3,2.26)
\psdots[dotsize=7pt 0,dotstyle=*](6.62,-0.92)
\rput[bl](7,-0.98){$=$ contracted vertices}
\rput[bl](7,3.18){Dual graph inside $S$}

\rput[bl](3,-0.98){distance $>n$}
\end{scriptsize}
\end{pspicture}

Let us consider the surface $W'$ obtained by contracting the irreducible components of the exceptional divisor of $W$ which are contracted to a point in $\cX$ and correspond to the vertices of the subgraph at distance $>n$ from the root. We denote by $\overline {f'}$ the induced map from $W'$ into $X$.
Let $K'$ be the corresponding graph in $W'$ (which is not necessarily a tree), again rooted at 
$o$ and endowed with the metric given by the distance from the root. 
Notice that since  $K$  is a tree and we have not changed anything at distance $<n$ there is a unique vertex $v_{n-1}$ in  $K'$ at distance 
$n-1$ from the root which is connected to $v_n$. 
Observe also that $(\overline {f'})^*\vert\D_J\vert$ 
either contains $E_{v_n}$ or
meets it in at most
$E_{v_n}\cap E_{v_{n-1}}$. Indeed, by the principal ideal theorem,
$(\overline {f'})^*\vert\D_J\vert$ is empty or pure co-dimension 1 while, by 
the inductive hypothesis, it is supported  at most in vertices a 
distance $\leq n$ from the root. 
On the other hand, if  $(\overline {f'})^*\vert\D_J\vert$ 
were supported on $E_{v_n}$, 
there would be a proper sub-stratum $X_{I'}$
containing $E_{v_n}$ such that 
$E_{v_n}$ meets $(\overline {f'})^{-1}\vert\D_{J'}\vert$ in 
at most the point $E_{v_n}\cap E_{v_{n-1}}$,
which contradicts \ref{prop:b&b} (b).
\end{proof}
The induction concluded, we return to the
graph $K$ in $W$, and colour $o$ and any
vertices contracted in $X_I$ black. As
such the sub-graph whose vertices are black
has a connected component $O\ni o$. Any
white vertex which meets $O$ in $K$ does
so along a unique edge- otherwise it would
admit two paths to $o$. As per the proof
of \ref{claim:ind} we pass to a singular
surface $W'$, but now the contraction of
all black vertices not in $O$, and again
the principal ideal theorem yields that
the pull-back of $\D_J$ can have support
at most in $O$. Consequently the whole
of the proper transform of $E$ except $O$
is in $X_I\backslash \vert\D_J\vert$, with the
(non-empty set) of white
vertices not meeting $O$ corresponding
to $\P^1$'s in the same, and otherwise
$\mathbb{A}^1$'s, and by \ref{prop:b&b} (a),
this remains true in $\cX_I\backslash\D_J$.
Finally introduce another graph by removing
$o$ and its unique edge and replacing
it by a  vertex, $o'$, 
for the proper transform
of the section, and an edge to the unique
exceptional curve that it meets, then this
is still a connected tree. Again we colour
$o'$ black, and let $O'$ be its
connected component in the black sub-graph. By
definition there is a white vertex meeting
$O'$, and every curve in $O'$ contracts
to a point in $f(\Gamma)$, so we're done.
\end{proof}

\begin{rmk}\label{rmk:15neri}
{\em In order to see the need whether for \ref{prop:b&b} (a) 
or the impossibility of replacing the moduli strata
in \ref{prop:b&b} in (b) by champ strata,
consider a situation as in the following picture.


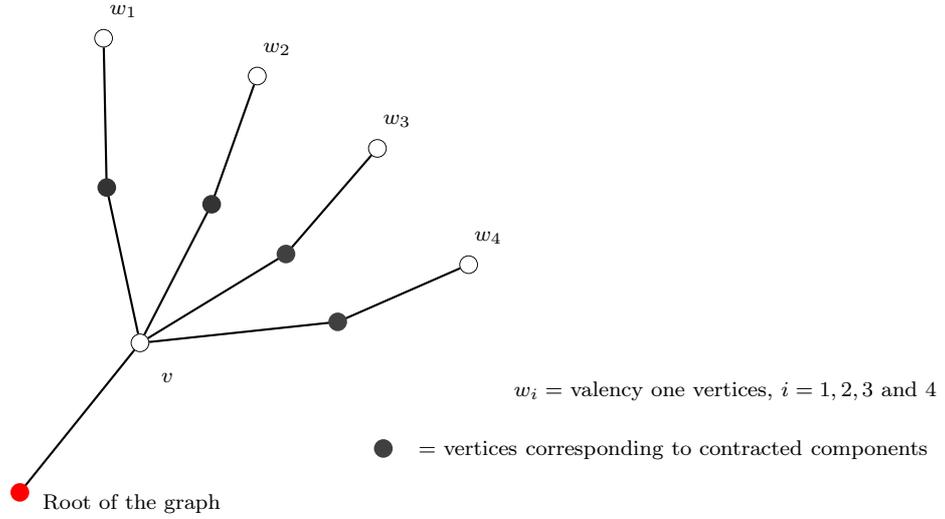
\newrgbcolor{tttttt}{0.2 0.2 0.2}
\psset{xunit=1.0cm,yunit=1.0cm,algebraic=true,dotstyle=o,dotsize=3pt 0,linewidth=0.8pt,arrowsize=3pt 2,arrowinset=0.25}
\begin{pspicture*}(-4,-2.5)(15.52,5.8)
\psline(-0.96,-1.08)(0.62,0.9)
\psline(0.62,0.9)(0.18,2.96)
\psline(0.62,0.9)(1.56,2.74)
\psline(0.62,0.9)(2.54,2.08)
\psline(0.62,0.9)(3.22,1.18)
\psline(3.22,1.18)(4.94,1.94)
\psline(2.54,2.08)(3.74,3.48)
\psline(1.56,2.74)(2.16,4.44)
\psline(0.18,2.96)(0.14,4.94)
\begin{scriptsize}
\psdots[dotsize=7pt 0,dotstyle=*,linecolor=red](-0.96,-1.08)
\rput[bl](-0.66,-1.36){{Root of the graph}}
\psdots[dotsize=7pt 0,dotstyle=*,linecolor=tttttt](0.18,2.96)
\psdots[dotsize=7pt 0,dotstyle=*,linecolor=tttttt](1.56,2.74)
\psdots[dotsize=7pt 0,dotstyle=*,linecolor=darkgray](2.54,2.08)
\psdots[dotsize=7pt 0,dotstyle=*,linecolor=darkgray](3.22,1.18)
\psdots[dotsize=7pt 0](0.14,4.94)
\rput[bl](0.22,5.22){$w_1$}
\psdots[dotsize=7pt 0](2.16,4.44)
\rput[bl](2.24,4.72){$w_2$}
\psdots[dotsize=7pt 0](3.74,3.48)
\rput[bl](3.82,3.76){$w_3$}
\psdots[dotsize=7pt 0](4.94,1.94)
\rput[bl](5.02,2.22){$w_4$}
\rput[bl](5.54,0.14){{$w_i =$ valency one vertices, $i=1,2,3$ and $4$}}
\psdots[dotsize=7pt 0,dotstyle=*,linecolor=darkgray](3.82,-0.50)
\rput[bl](4.28,-0.64){{$=$ vertices corresponding to contracted components }}
\psdots[dotsize=7pt 0](0.62,0.9)
\rput[bl](0.9,0.38){$v$}
\end{scriptsize}
\end{pspicture*}

This can be obtained, for example, by blowing up points 
on a surface with the vertex $v$ of
valency 5 being occasioned by the first
blow up. As such the black vertices are
-2 curves, the other white vertices, $w_i$, -1,
and the former may be blown down to 
non-scheme like points on a smooth
champ. Equally, on this champ the
picture, on adding the red vertex,
can be obtained as a degeneracy of
a family of curves blowing down a
section to a point $p$ which meets the picture in
$v$ alone. 
Each vertex of valency 1 is a parabolic
champ, but the vertex which passes through
$p$ is not. This vertex could also be
the boundary, so one cannot replace
moduli by champ in \ref{prop:b&b} (b) either.\hfill$\square$}
\end{rmk}

Thus without \ref{prop:b&b} (a) \& (b)
one adds nothing to \ref{prop:2} beyond quoting
the degree bound \ref{eq:degree} of Miyaoka-Mori. 
Nevertheless, we can continue the blatant
plagerism of the same by way of,

\begin{prop}\label{thm:logMM}
Let $\cX\to X$ be a champ over a 
projective variety, and $\D=\sum_i\D_i$
the irreducible components of an
effective Weil divisor on $\cX$,
with $H$ nef. on $X$. Suppose further
that for some stratum $\cX_I$ each
$\D_j$, $j\in J$ is $\mathbb{Q}$-Cartier,
and that there are enough deformations
in $\cX_I\backslash\D_J$ at every curve
whose image is some $f:C\to\cX_I$, then
should $(K_{\cX}+\D)\cdot_f C$ be negative
there is a parabolic champ $\cL$
in some stratum $\cX_{I'}\backslash\D_{J'}$,
$I'\supseteq I$ such that,
$$
 H\cdot \cL \leq 2\dim(\cX_I)\frac{(H\cdot_f C)}{-(K_{\cX}+\D)\cdot_f C}.
$$
If, in addition, \ref{prop:b&b} (a)\& (b) hold, then
for generic $x\in C$ there is such a champ
$\cL\ni x$ parabolic in $\cX_I\backslash\D_J$.
\end{prop}
\begin{proof}
Whenever the normalisation $\cC$ of the image
of $f$ is not parabolic this follows from
\cite[Theorem 5]{MM} with the same proof
by \ref{lem:dim2} and \ref{prop:b&b}. 
Otherwise, to
make
the trick work of replacing $f$ by a composition
with the geometric Frobenius $C^{(1)}\to C$
modulo large primes, or a power of the same
if $k$ already has positive characteristic,
we need that there are more deformations of
$C\to\cX_I$ then there are of $C\to \cC$. This
amounts to,
$$(K_{\cX}+\D).\cC < -2$$
and once one has this, one argues exactly
as in the non-parabolic case, while if this fails 
the above bound is
at least $H\cdot \cC$, and one takes $\cL=\cC$.
\end{proof}

%
\section{The cone theorem}
%

Cone theorems are weaker than \ref{thm:logMM},
so we need less, {\it e.g.} the following variant:

\begin{prop}\label{prop:breakP1}
Let $\cX\to X$ 
be a champ over a projective variety
with snc boundary $\D=\sum_i \D_i$ on $\cX$,
and, bearing in mind \ref{conv:X_I}, $\cP$ a parabolic champ in some
$\cX_{I}\backslash\D_J$, then:
$$
\cP \sim Z+ \cL
$$
where $Z$ is a $1$-cycle on $\cX$, $\cL$
is parabolic in
$\cX_{I'}\backslash\D_{J'}$ for some $I'\supset I$, and satisfies:  
$$
-\cL\cdot (K_{\cX} +\D)\leq 2\dim(X),
$$
\end{prop}

\begin{proof}
If $-\cP\cdot(K_{\cX}+\D)\leq 2 \dim(\cX)$ we set $Z=0$
and $\cP=\cL$. So we assume 
$-\cP\cdot (K_{\cX_{I}}+\D_{J})>2 \dim(\cX)$. 
Take a covering $\P^1\to \cP$ of degree $d\in\mathbb{N}$
over the moduli.
By Lemma \ref{lem:dim2} we get:
\begin{equation}\label{eq:>0}
 \dim_{[f]} \Mor(\P^1,\cX_I,\D_J)> (1+2d)\cdot \dim(\cX_I).
\end{equation}

We fix an ample divisor $H$ on $X$ and argue by induction on 
the degree $\vert\cP\vert\cdot H\in\mathbb{N}$ of the moduli.
As in the proof of \ref{thm:logMM}, we need more
deformations in $\cX_I$ than there are map 
from $\P^1\to\vert\cP\vert$, 
{\it i.e} $2d+1$. This we have, so
fixing $2d+1$
points in a deformation, \ref{prop:b&b}
yields a parabolic champ $\cL$ in some
$\cX_{I'}\backslash \D_{J'}$, $I'\supset J'$ satisfying,
$$
H\cdot \vert\cL\vert\leq 2d\cdot \frac{H\cdot \vert\cP\vert}{1+2d}< H\cdot \vert\cP\vert 
$$
If the degree of $\vert\cP\vert$ is 1, this is nonsense,
and we have
the base of the induction. Otherwise, $\vert\cL\vert$
has smaller degree, and we conclude.
\end{proof}

Let us observe,

\begin{rmk}\label{rmk:bounds}{\em
The bound $2\dim X$ is far from optimal,
and it's fastidious rather than difficult to
do better. In the wholly scheme like case
the optimal bound $\dim X +1$ is known
to hold, \cite{K}, but one has to vary
the construction from \cite{MM}. However,
even without varying it, things actually
get better rather than worse the less
scheme like the situation becomes since,
in general,
there are much fewer maps from a $\P^1$
to a parabolic champ than there are to
its moduli.}
\end{rmk}

In any case the following deduction from \ref{prop:breakP1}
and \cite[Theorem III.1.2]{K} is formal,

\begin{prop}\label{thm:cone}
Let $\cX\to X$ be a smooth champ over
a projective variety, and $\D=\sum_i\D_i$  a snc
divisor on $\cX$, then
there exists a countable family $\{L_k\subset X\}$ 
of curves whose induced champs $\cL_k\to L_k$
are parabolic in some $\cX_{I_k}\backslash\D_{J_k}$, satisfying,
$$
 0<-(K_{\cX}+\D)\cdot \cL_k\leq  2\cdot \dim(X),
$$
such that  
\begin{equation}\label{eq:cone} 
\overline {\mathrm{NE}}(\cX)=
 \overline {\mathrm{NE}}(\cX)_{(K_\cX+\D)\geq 0} + \sum_k \R_{+} [\cL_k]
\end{equation}
and the parabolic rays $\R_{+}[\cL_k]= \R_{+} [L_k]$ 
are locally discrete in $\mathrm{N}_1(\cX)_{(K_\cX+\D)< 0}$.
\end{prop}
\begin{proof}
Since we have somewhat weaker theorems at our
disposition, {\it e.g.} \ref{prop:breakP1}
rather than \cite[Theorem II.5.7]{K}
we'll go through the proof of III.1.2 
in op. cit. to check that it works.

The first stage is to note that we
have countably many classes $L_k$ such that, 
\begin{enumerate}
\item[(i)] $ 0<-(K_{\cX}+\D)\cdot L_k\leq  2\cdot \dim(\cX)$;
\item[(ii)] We may identify $L_k$ with a curve such that the
champ $\cL_k$ over the same is parabolic in some $\cX_{I_k}\backslash\D_{J_k}$. 
\end{enumerate}

Next form the closed cone $W$ of the right hand side
of \ref{eq:cone}, and suppose this is not the left
hand side, then there is a linear functional $M$
non-negative on the left hand side, strictly 
positive on $W\backslash\{0\}$, but vanishing
on some $0\neq z\in \overline {\mathrm{NE}}(X)$. 
In particular $-(K_{\cX}+\Delta)\cdot z>0$.

Now, say,
$C_i=\sum_{j} a_{ij} C_{ij}$ are effective cycles
limiting on $z$, then for every $i$ sufficiently 
large there is a
$j$ such that,
\begin{equation}\label{eq:copy}
\dfrac{M\cdot C_{ij}}{-(K_{\cX}+\Delta)\cdot C_{ij}}
\leq
\dfrac{M\cdot C_{i}}{-(K_{\cX}+\Delta)\cdot C_{i}}
\end{equation}

By \ref{thm:logMM} applied to the $C_{ij}$ 
and \ref{prop:breakP1} applied to the
resulting parabolic champs $\cP_{ij}$,
there are champs $\cL_{k(i)}$ parabolic
in some $\cX_{I_{k(i)}}\backslash\D_{J_{k(i)}}$
such that,
$$
M\cdot \vert\cL_{k(i)}\vert \leq 2\dim (X)\cdot \dfrac{M\cdot C_i}{-(K_{\cX}+\D)\cdot C_i}
$$
and $-(K_{\cX}+\Delta)\cdot\cL_{k(i)}<2\cdot\dim (X)$.
By definition, the classes $[\cL_{k(i)}]$ are
integral and belong
to $W$, so for $i$ sufficiently large the
left hand side is bounded below independent
of $i$, while the right hand side tends to
zero. This is nonsense, so, indeed,
$W=\overline {\mathrm{NE}}(X)$.

To conclude we need to know that the
parabolic rays are locally discrete
and the right hand side of \ref{eq:cone}
is a closed cone. These statements,
however, now follow verbatim as in
\cite[Theorem III.1.2]{K} up
to the simple expedient of replacing
$K_X$ in op. cit. by $K_{\cX}+\Delta$.
\end{proof}

%

\vskip 10pt

\noindent
{\small Michael McQuillan\\
Universit\`a degli Studi di Roma "Tor Vergata"\\
Viale della Ricerca Scientifica\\

\smallskip

\noindent
{\small  Gianluca Pacienza\\
Institut de Recherche Math\'ematique Avanc\'ee, 
Universit\'e de Strasbourg et CNRS\\ 
7, Rue R. Descartes - 67084 Strasbourg Cedex, France \\
E-mail : {\tt pacienza@math.unistra.fr}}

\end{document}